\documentclass[11pt,oneside]{article}
\usepackage{amsfonts,amsmath,amsthm,mathrsfs,indentfirst,amssymb}
\usepackage[mathscr]{euscript}
\usepackage{xcolor}
\usepackage{enumitem}
\newcommand{\lbl}[1]{\label{#1}}

\newtheorem{theo}{Theorem}[section]

\newtheorem{lem}{Lemma}[section]
\newtheorem{remark}{Remark}[section]

\numberwithin{equation}{section}

\newcommand{\be}{\begin{equation}}
\newcommand{\ee}{\end{equation}}
\newcommand\bes{\begin{eqnarray}} \newcommand\ees{\end{eqnarray}}
\newcommand{\bess}{\begin{eqnarray*}}
\newcommand{\eess}{\end{eqnarray*}}

\newcommand\ep{\varepsilon}
\newcommand\rh{\rightarrow}
\newcommand\kk{\left}
\newcommand\rr{\right}
\newcommand\dd{\displaystyle}

\newcommand\df{\dd\frac}
\newcommand\ap{\alpha}

\newcommand\yy{\infty}

\newcommand\R{\mathbb{R}}

\newcommand\ud{\underline}

\newcommand\bu{\bar u}
\newcommand\bv{\bar v}
\newcommand\bh{\bar h}

\newcommand\hw{\hat w}
\newcommand\hz{\hat z}
\newcommand\hh{\hat h}
\newcommand\oa{\Omega_{1T}}
\newcommand\ob{\Omega_{2T}}
\newcommand\oc{\Omega_{3T}}
\newcommand\ot{\Omega_T}
\newcommand\dt{\Delta_T}
\def\theequation{\arabic{section}.\arabic{equation}}
\begin{document}
\setlength{\baselineskip}{16pt} \pagestyle{myheadings}

\begin{center}{\Large\bf The diffusive Beddington-DeAngelis predator-prey model}\\[2mm]
{\Large\bf with nonlinear prey-taxis and free boundary}\footnote{This work was
supported by NSFC Grants 11371113 and 11771110}\\[4mm]
 {\Large  Jianping Wang, \ \ Mingxin Wang\footnote{Corresponding author. {\sl E-mail}: mxwang@hit.edu.cn}}\\[1mm]
{Department of Mathematics, Harbin Institute of Technology, Harbin 150001, PR China}
\end{center}

\begin{quote}
\noindent{\bf Abstract.} The diffusive Beddington-DeAngelis predator-prey model
with nonlinear prey-taxis and free boundary is considered. We investigate the existence and uniqueness, regularity and uniform estimates, and long time behavior of the global solution. Some sufficient conditions for both spreading and vanishing are established.


 \end{quote}

 \section{Introduction}
\setcounter{equation}{0} {\setlength\arraycolsep{2pt}

The dynamical relationship between the predator and prey has been investigated widely in recent years due to its universal existence and importance in mathematical biology and ecology.
Let $\Omega$ be a bounded domain in $\mathbb{R}^n$ ($n\geq1$). After rescaling, the diffusive predator-prey model with Beddington-DeAngelis functional response takes the form \cite{Be, DGO}:
   \bess
 \left\{\begin{array}{lll}
 u_t-u_{xx}=-au+\dd\frac{buv}{c+u+mv}, &x\in\Omega,\ \ t>0,\\[4mm]
 v_t-dv_{xx}=v(q-v)-\dd\frac{ruv}{c+u+mv},\ \ &x\in\Omega,\ \ t>0,
 \end{array}\right.
 \eess
where where $u$ and $v$ represent predator and prey densities, respectively. Constants $a,b,c, m, q, r,d$ are positive, $a$ is the mortality rate of the predator which does not depend on the prey density, the function $\frac{rv}{c+u+mv}$ is the Beddington-DeAngelis functional response, $b/r$ is the conversion rate from prey to predator.

In the above model, the predator and prey species are usually assumed to move randomly in their habitat. It has been recognized that in the spatial predator-prey interaction, in addition to the random diffusion of predator and prey, the spatiotemporal variations of the predator velocity are affected by the prey gradient \cite{ABM,Kare,Lee}. The diffusive predator-prey model with Beddington-DeAngelis functional response and prey-taxis reads as follows
   \bes
   \left\{\begin{array}{lll}
 u_t-u_{xx}+(u\chi(u)v_x)_x=-au+\dd\frac{buv}{c+u+mv}, \ \ &x\in\Omega,\ \ t>0,\\[4mm]
 v_t-dv_{xx}=v(q-v)-\dd\frac{ruv}{c+u+mv},\ \ &x\in\Omega,\ \ t>0,
 \end{array}\right.
 \label{1.0}
 \ees
In this model, the predator is attracted by the prey and $\chi$ denotes their prey-tactic sensitivity and satisfies
\bes\left\{\begin{array}{lll}
\chi(u)\in C^1([0,\infty)),\ \chi(u)=0\, \ {\rm for}\ u\geq u_m,\ {\rm and}\ \chi'(u)\ {\rm is\ Lipschitz\ continuous},\ {\rm i.e.},\\
|\chi'(u_1)-\chi'(u_2)|\leq L|u_1-u_2|\, \ {\rm for\ any}\ u_1, u_2\in[0,\infty),
 \end{array}\right.\label{1.2}
 \ees
where $u_m$ and $L$ are two positive constants.

The assumption that $\chi(u)\equiv0$ for $u\geq u_m$ has a clear biological interpretation \cite{ABM}: the predator stops to accumulate at a given point after its density attains a certain threshold value $u_m$ and the prey-tactic cross diffusion $\chi(u)$ vanishes identically when $u\geq u_m$. The assumption that $\chi'(u)$ is Lipschitz continuous is a regularity requirement for our qualitative analysis. Let $\eta(u)=u\chi(u)$, then it is easily seen that
 \begin{enumerate}[leftmargin=2em]
\item[$\bullet$] \, $\eta(u)$ and $\eta'(u)$ are bounded, and $\eta'(u)$ is Lipschitz continuous.
\end{enumerate}

In many realistic modeling situations, both the predator and prey have a tendency to emigrate from the boundary to obtain their new habitat and to improve the living environment
(\cite{Wjde14, wang2015, WZhang, WZjdde17, ZW2014, ZLC}). Hence it is more reasonable to consider the domain with a free boundary. As a general rule, to avoid being hunted, the prey will have a stronger tendency than the predator. So, we may consider that the free boundary is caused only by the prey.
For simplicity, we assume that the species only spreads to the right (the left boundary is fixed) in a one-dimensional environment. According to the deductions of the free boundary conditions in \cite{BDK} and \cite{WZjdde17}, a free boundary problem related to \eqref{1.0} can be written as
  \bes
 \left\{\begin{array}{lll}
 u_t-u_{xx}+(u\chi(u)v_x)_x=\dd\frac{buv}{c+u+mv}-au:=f(u,v),\ \ &t>0,\ \ 0<x<h(t),\\[4mm]
 v_t-dv_{xx}=v(q-v)-\dd\frac{ruv}{c+u+mv}:=g(u,v),\ \ &t>0, \ \ 0<x<h(t),\\[2mm]
 u_x=v_x=0,\ \ \ &t\ge 0,\ \ x=0,\\[1mm]
  u=v=0, \ \ \ &t\ge0,\ \ x=h(t),\\[1mm]
 h'(t)=-\mu v_x(t,h(t)), \ &t\ge0,\\[1mm]
 h(0)=h_0,\\[1mm]
 u(0,x)=u_0(x),\ \ v(0,x)=v_0(x),\ \ &0\leq x\leq h_0,
 \end{array}\right.\label{1.1}
 \ees
where $h_0, \mu$ are given positive constants. In the problem \eqref{1.1} it is assumed that the free boundary is caused only by the prey.
 %The derivation of such free boundary condition in population dynamics can be
 %found in \cite{BDK, WZjdde17} from biological viewpoint and \cite{HIMN}
 %using a singular limit analysis.

It is interesting to investigate the dynamics of the problem \eqref{1.1} because of the influence of chemotactic cross-diffusion. We will show that \eqref{1.1} has a unique solution and $\dd\lim_{t\to\infty} h(t)\leq\infty$. Moreover, if $\dd\lim_{t\to\infty} h(t)<\infty$, then the predator and prey fails to establish and vanishes eventually. Some sufficient conditions for spreading and vanishing are given.

Throughout this paper we also assume that $0<\alpha<1$ and
 \bes\left\{\begin{array}{l}
 u_0\in C^{2+\alpha}([0,h_0]),\  u'_0(0)=0,\  u_0(x)\geq,\,\not\equiv 0\ \ {\rm in}\, \
[0,h_0), \ \ u_0(x)=0\ \ {\rm in}\, \ [h_0,\yy),\\[.5mm]
v_0\in C^{2+\alpha}([0,h_0]),\  v'_0(0)=0,\  v_0(x)\geq,\,\not\equiv 0\ \ {\rm in}\, \
[0,h_0), \ \ v_0(x)=0\ \ {\rm in}\, \ [h_0,\yy).
 \end{array}\right.\label{1.3}
 \ees

This article is divided into five sections. Section 2 proves the global existence and uniqueness of solutions. Section 3 establishes an important estimate for the solution. Section 4 is devoted to understanding the behavior of the solution. Section 5 provides some sufficient conditions for spreading and vanishing.

Before ending this section we should mention that the free boundary problems of the diffusive competition models have been studied widely by many authors. Please refer to \cite{CLW, dulin2013, DWZ15, GW, GW15, WjZ, WZzamp, Wcom15, WZjdde14, Wu15, zhaowang2014, ZZL}.

\section{Existence and uniqueness}

For $T>0$ we set $D_T=[0, T]\times[0, h(t)]$ and $\dt:=[0,T]\times[0,1]$.

\begin{theo}\label{t2.1}
Assume that $\chi(u)$ and $(u_0,v_0)$ satisfy the conditions \eqref{1.2} and \eqref{1.3},
respectively.

{\rm(i)}\, Then there is a $T>0$ such that the problem \eqref{1.1} admits a unique solution
\bes
(u,v,h)\in [C^{1+\frac{\alpha}2,2+\alpha}(D_T)]^2\times C^{1+\frac{1+\alpha}2}([0,T]), \label{2.1}
\ees
and $h'(t)>0$ in $(0,T]$ and
\bes
\|u,\,v\|_{C^{1+\frac{\alpha}2,2+\alpha}(D_T)}+\|h\|_{C^{1+\frac{1+\alpha}2}([0,T])}\leq C, \label{2.2}
\ees
where positive constants $T$ and $C$ depend on $\|u_0(x)\|_{C^{2+\alpha}([0,h_0])}$, $\|v_0(x)\|_{C^{2+\alpha}([0,h_0])}$, $h_0$ and $h^*:=-\mu v_0'(h_0)$.

{\rm(ii)}\, Let $0<\tau<\infty$ and $(u,v,h)\in [C^{1+\frac{\alpha}2,2+\alpha}(D_{\tau})]^2\times C^{1+\frac{1+\alpha}2}([0,\tau])$ be the unique solution of \eqref{1.1}. Then there exist positive constant $M_i$, $i=1,2,3$, independent of $\tau$, such that
 \bes
&0<u\leq M_1,\ \ 0< v\leq M_2\ \ &{\rm in}\ \ (0, \tau]\times[0, h(t)), \label{2.5}\\[1mm]
&0<h'(t)\leq M_3\ \ &{\rm in}\ \ (0, \tau].\label{2.7}
 \ees
\end{theo}

\begin{proof}
The claims concerning local-in-time existence of the classical solution to the problem \eqref{1.1} are well established by a fixed point theorem. The proof is quite standard, we refer readers to \cite{Wjfa16, WZjdde17}.

First of all, in order to straighten the free boundary, we define
\[ y=x/h(t),\ \ w(t,y)=u(t,h(t)y),\ \ z(t,y)=v(t,h(t)y).\]
Then it follows from $\eqref{1.1}$ that $w(t,y),z(t,y),h(t)$ satisfy
\bes
 \left\{\begin{array}{lll}
 w_t-\zeta(t)w_{yy}-\xi(t)yw_y+\zeta(t)\eta'(w)z_y w_y\\[1mm]
 \qquad=-\zeta(t)\eta(w)z_{yy}+f(w,z),\ \ &t>0, \ \ 0<y<1,\\[1mm]
 z_t-d\zeta(t)z_{yy}-\xi(t)yz_y=g(w,z),\ \ &t>0, \ \ 0<y<1,\\[1mm]
  w_y(t,0)=w(t,1)=z_y(t,0)=z(t,1)=0,\ \ \ &t\ge0,\\[1mm]
 w(0,y)=u_0(h_0y):=w_0(y),\ \ \ &0\leq y\leq 1,\\[1mm]
  z(0,y)=v_0(h_0y):=z_0(y),\ \ \ &0\leq y\leq 1,\\[1mm]
 h'(t)=-\mu \frac{1}{h(t)}z_y(t,1),\ \ t\geq 0;\ \ h(0)=h_0,
 \end{array}\right.\label{21a}
\ees
where $\zeta(t)=h^{-2}(t)$ and $\xi(t)=h'(t)/h(t)$.

{\it Step 1: The existence}. We prove the existence result by a fixed point theorem. Let $h^*:=-\mu v_0'(h_0)$. For $0<T\leq \min\big\{1,\frac{h_0}{2(1+h^*)}\big\}$, we define
\bess
&\oa:=\{ w\in C^{\ap/2,\ap}(\dt):\ w(0,y)=w_0(y),\ w\ge0,\ \|w-w_0\|_{C^{\ap/2,\ap}(\dt)}\le 1 \},\\
&\ob:=\{ z\in C^{\ap/2,\ap}(\dt):\ z(0,y)=z_0(y),\ z\ge0,\ \|z-z_0\|_{C^{\ap/2,\ap}(\dt)}\le 1 \},\\
&\oc:=\{ h\in C^{1+\ap/2}([0,T]):\ h(0)=0,\ h'(0)=h^*,\ \|h'-h^*\|_{C^{\ap/2}([0,T])}\le1 \}.
\eess
Clearly, $\ot:=\oa\times\ob\times\oc$ is a bounded and closed convex set of  $[C^{\ap/2,\ap}(\dt)]^2\times C^{1+\ap/2}([0,T])$. For any $h\in \oc$, we have
\[ |h(t)-h_0|\leq T\|h'\|_\yy\leq T(1+h^*)<\frac{h_0}2, \]
which yields that
\[ \frac{h_0}2\le h(t)\le \frac{3h_0}2,\ \ \forall\ t\in[0,T]. \]
Therefore the transformation $(t,x)\rh (t,y)$ introduced at the beginning of the proof is well defined. Denote
 \[\Lambda=\{\|u_0,v_0\|_{C^{\ap}([0,h_0])}, \, h_0, \, h^*\}.\]
For the given $(w,z,h)\in\ot$. Direct computations entails that there exists a constant $C_1>0$ depends on $\Lambda$ such that
\[\|\zeta(t)\|_{C^{\ap/2,\ap}(\dt)}+\|\xi(t)\|_{C^{\ap/2,\ap}(\dt)}+\|g(w,z)\|_{C^{\ap/2,\ap}(\dt)}\le C_1.\]
Recall that $z_0(y):=v_0(h_0y)\in C^{2+\ap}([0,1])$. The parabolic Schauder theory asserts that the problem
 \bess
 \left\{\begin{array}{lll}
 \hz_t-d\zeta(t)\hz_{yy}-\xi(t)y\hz_y=g(w(t,y),z(t,y)),\ \ &0<t\le T, \ \ 0<y<1,\\[1mm]
\hz_y(t,0)=\hz(t,1)=0,\ \ \ &0\le t\le T,\\[1mm]
 \hz(0,y)=z_0(y),\ \ \ &0\leq y\leq 1
 \end{array}\right.
\eess
 has a unique solution $\hz\in C^{1+\ap/2,2+\ap}(\dt)$ and
 \bes
  \|\hz\|_{C^{1+\ap/2,2+\ap}(\dt)}\le C_2 \lbl{24a}
  \ees
for some $C_2>0$ which is a constant dependent on $\Lambda$ and $\|v_0\|_{C^{2+\ap}([0,h_0])}$. Moreover, by the maximum principle we infer that $\hz>0$ in $(0,T]\times[0,1)$. Hence, $\hz_y(t,1)<0$ for $t\in(0,T]$ by the Hopf boundary lemma. Thus the problem
 \[ \hh'(t)=-\mu \frac{1}{h(t)}\hz_y(t,1),\ \ 0<t\le T; \qquad \hh(0)=h_0\]
has a unique solution $\hh(t)$. Obviously, $\hh'(0)=h^*$, $\hh'(t)>0$ in $(0,T]$ and
\bes
 \hh'\in C^{\frac{1+\ap}2}([0,T])\ \ {\rm and}\ \  \|\hh'\|_{C^{\frac{1+\ap}2}([0,T])}\le C_3,\lbl{25a}
 \ees
where $C_3>0$ is dependent on $\Lambda$ and $\|v_0\|_{C^{2+\ap}([0,h_0])}$.

Now let us consider the following problem
\bes
 \left\{\begin{array}{lll}
 \hw_t-\zeta(t)\hw_{yy}+[\zeta(t)\eta'(w)\hz_y-\xi(t)y]\hw_y+\zeta(t)\chi(w)\hz_{yy}\hw\\[1mm]
 \qquad=f(w(t,y),z(t,y)),\quad \ 0<t\le T, \ \ 0<y<1,\\[1mm]
 \hw_y(t,0)=\hw(t,1)=0,\qquad \, \ 0\le t\le T,\\[1mm]
 \hw(0,y)=w_0(y),\qquad \qquad \quad 0\leq y\leq 1.
 \end{array}\right.\lbl{26a}
\ees
In view of \eqref{24a} and assumption \eqref{1.2}, it holds that
\[ \|\zeta(t)\eta'(w)\hz_y-\xi(t)y\|_{C^{\ap/2,\ap}(\dt)}+ \|\zeta(t)\chi(w)\hz_{yy}\|_{C^{\ap/2,\ap}(\dt)}+\|f(w,z)\|_{C^{\ap/2,\ap}(\dt)}\le C_4 \]
for some positive constant $C_4$ which depends on $\Lambda$,  $\|v_0\|_{C^{2+\ap}([0,h_0])}$, $\|u_0\|_{C^{\ap}([0,h_0])}$ and $L$, where $L$ is given by \eqref{1.2}. Recalling $\hw_0(y):=u_0(h_0y)\in C^{2+\ap}([0,1])$ and using parabolic Schauder theory, the problem \eqref{26a} admits a unique solution $\hw\in C^{1+\ap/2,2+\ap}(\dt)$ and
\bes
 \|\hw\|_{C^{1+\ap/2,2+\ap}(\dt)}\le C_5, \lbl{27a}
\ees
where $C_5$ is a positive constant depending on $\|u_0,v_0\|_{C^{2+\ap}([0,h_0])}, h_0,h^*$ and $L$. Furthermore, it follows from the maximum principle that $\hw>0$ in $(0,T]\times[0,1)$.

Based on the above analysis, we can define
$$F:\ot\rh [C^{\ap/2,\ap}(\dt)]^2\times C^{1+\ap/2}([0,T])$$
by
\[ F(w,v,h)=(\hw,\hz,\hh). \]
Evidently, $(w,v,h)\in\ot$ is a fixed point of $F$ if and only if it solves \eqref{21a}. From \eqref{24a}, \eqref{25a} and \eqref{27a}, we see that $F$ is compact and
\bess
&\|\hw-w_0\|_{C^{\ap/2,\ap}(\dt)}\le C_6\max\{T^{\frac{\ap}2},T^{1-\frac{\ap}2}\}\|\hw\|_{C^{1+\frac{\ap}2,2+\ap}(\dt)}\le C_6C_5\max\{T^{\frac{\ap}2},T^{1-\frac{\ap}2}\},\\
&\|\hz-z_0\|_{C^{\ap/2,\ap}(\dt)}\le C_6\max\{T^{\frac{\ap}2},T^{1-\frac{\ap}2}\}\|\hz\|_{C^{1+\frac{\ap}2,2+\ap}(\dt)}\le C_6C_2\max\{T^{\frac{\ap}2},T^{1-\frac{\ap}2}\},\\
&\|\hh'-h^*\|_{C^{\ap/2}([0,T])}\le C_7T^{\frac{1}2}\|\hh'\|_{C^{\frac{1+\ap}2}([0,T])}\le C_7C_3T^{\frac{1}2},
\eess
where $C_6,C_7>0$ is independent of time $T$. Hence $F$ maps $\ot$ into itself if $T$ is small enough. Consequently $F$ has at least one fixed point $(w,z,h)\in\ot$, and then \eqref{21a} admits at least one solution $(w,z,h)$ defined in $[0,T]$. Moreover, $w,z>0$ in $(0,T]\times[0,1)$. Noticing that $z(t,1)=0$, it deduces by the Hopf boundary lemma that $z_y(t,1)<0$, which implies $h'(t)>0$ for $t\in(0,T]$.

{\it Step 2: The uniqueness}. Let $(w_i,z_i,h_i),i=1,2$ be two local solutions of \eqref{21a}, which are defined for $t\in[0,T]$ with $0<T\ll 1$. By the maximum principle and Hopf boundary lemma, we have
\[ w_i,z_i>0\ \ {\rm for}\ t\in(0,T],\ x\in[0,1), \]
and
\[ h_i'(t)=-\mu z_{i,y}(t,1)>0 \ \ {\rm for}\ t\in(0,T]. \]
Thereby, we may assume that
\[ h_0\leq h_i(t)\le h_0+1\ \ {\rm in}\ [0,T],\ i=1,2. \]
Denote $\zeta_i(t)=h_i^{-2}(t),\xi_i(t)=h_i'(t)/h_i(t)$. Set
 \[w=w_1-w_2, \ \ \ z=z_1-z_2, \ \ \ h=h_1-h_2.\]
It then follows from \eqref{21a} that $z$ solves the problem
 \bess
 \left\{\begin{array}{lll}
 z_t-d\zeta_1(t)z_{yy}-\xi_1(t)yz_y-A_1(t,y)z=A(t,y),\ \ &0<t\le T, \ \ 0<y<1,\\[1mm]
z_y(t,0)=z(t,1)=0,\ \ \ &0\le t\le T,\\[1mm]
 z(0,y)=0,\ \ \ &0\leq y\leq 1,
 \end{array}\right.
\eess
where
\bess
A(t,y)=dz_{2,yy}[\zeta_1(t)-\zeta_2(t)] +z_{2,y}[\xi_1(t)y-\xi_2(t)y]+z_2(A_1(t,y)-A_2(t,y))
\eess
and
\bess
A_i(t,y)= q-z_i-\dd\frac{rw_i}{c+w_i+mz_i}, \ \ \ i=1,2.
\eess
It is easy to see that $\zeta_1,\xi_1,A_1\in L^\yy(\dt)$. Applying the parabolic $L^p$ theory and Sobolev embedding theorem we derive that
\begin{align}
\|z\|_{C^{\frac{1+\ap}2,1+\ap}(\dt)}+\|z\|_{W^{1,2}_p(\dt)}\le C_8\|A\|_{L^p(\dt)}
 %\le& C_8(d\|z_{2,yy}(\zeta_1(t)-\zeta_2(t))\|_{L^p(\dt)} %+\|z_{2,y}(\xi_1(t)y-\xi_2(t)y)\|_{L^p(\dt)}\nonumber\\
  %&+\|z_2(A_1(t,y)-A_2(t,y))\|_{L^p(\dt)})\nonumber\\
\le C_9\big(\|h\|_{C^1([0,T])}+\|w,\,z\|_{L^p(\dt)}\big),\lbl{28a}
\end{align}
where $C_8,C_9$ are independent of $T$. Notice that $z(0,y)=0$, it follows that
\begin{align}
\|z\|_{C^{\frac{\ap}2,0}(\dt)}=& \|z-z_0\|_{C(\dt)}+[z-z_{0}]_{C^{\frac{\ap}2,0}(\dt)} \nonumber \\
\le& T^{\frac{\ap}2}\|z\|_{C^{\frac{\ap}2,0}(\dt)}+T^{\frac{1}2}\|z\|_{C^{\frac{1+\ap}2,0}(\dt)}\nonumber\\
\le& (T^{\frac{\ap}2}+T^{\frac{1}2})\|z\|_{C^{\frac{1+\ap}2,1+\ap}(\dt)}\nonumber\\
\le& 2C_{9}T^{\frac{\ap}2}\big(\|h\|_{C^1([0,T])}+\|w,\,z\|_{L^p(\dt)}\big).\lbl{2.13}
\end{align}

Similarly, $w$ satisfies
\bess
 \left\{\begin{array}{lll}
 w_t-\zeta_1(t)w_{yy}+B_1(t,y) w_y-D_1(t,y)w =E(t,y),\ \ &0<t\le T, \ \ 0<y<1,\\[1mm]
 w_y(t,0)=w(t,1)=0,\ \ \ &0\le t\le T,\\[1mm]
 w(0,y)=0,\ \ \ &0\leq y\leq 1,
 \end{array}\right.
\eess
where
\bess
E(t,y)=w_{2,yy}[\zeta_1(t)-\zeta_2(t)]-w_{2,y}\kk[B_1(t,y)-B_2(t,y)\rr]+w_2\kk[D_1(t,y)-D_2(t,y) \rr],
\eess
and
\bess
B_i(t,y)&=&\zeta_i(t)\eta'(w_i)z_{i,y}-\xi_i(t)y,\\[1mm] D_i(t,y)&=&-\zeta_i(t)\chi(w_i)z_{i,yy}-a+\dd\frac{bz_i}{c+w_i+mz_i},
\eess
$i=1,2$.
Evidently, $\zeta_1,B_1,D_1\in L^\yy(\dt)$. Notice \eqref{28a}. Employing the parabolic $L^p$ theory and embedding theorem we derive
\begin{align}
\|w\|_{C^{\frac{1+\ap}2,1+\ap}(\dt)}\le& C_{10}\|E\|_{L^p(\dt)}\nonumber\\
\le& C_{10}\big(\|w_{2,yy}(\zeta_1-\zeta_2)\|_{L^p(\dt)} +\|w_{2,y}(B_1-B_2)\|_{L^p(\dt)}
+\|w_2(D_1-D_2)\|_{L^p(\dt)}\big)\nonumber\\
\le& C_{11}\big(\|h\|_{C^1([0,T])}+\|w\|_{L^p(\dt)}+\|z\|_{W^{1,2}_p(\dt)}\big)\nonumber\\
\le& C_{12}\big(\|h\|_{C^1([0,T])}+\|w,\,z\|_{L^p(\dt)}\big),\nonumber
\end{align}
where the constants $C_{10},C_{11},C_{12}$ are independent of $T$. Notice that $w(0,y)=0$, it follows that
\begin{align}
\|w\|_{C^{\frac{\ap}2,0}(\dt)}=& \|w-w_0\|_{C([0,T])}+[w-w_0]_{C^{\frac{\ap}2,0}([0,T])} \nonumber \\
\le& \|w\|_{C^{\frac{\ap}2,0}(\dt)}T^{\frac{\ap}2}+\|w\|_{C^{\frac{1+\ap}2,0}(\dt)}T^{\frac{1}2} \nonumber\\
\le& \|w\|_{C^{\frac{1+\ap}2,1+\ap}(\dt)}(T^{\frac{\ap}2}+T^{\frac{1}2})\nonumber\\
\le& 2C_{12}T^{\frac{\ap}2}(\|h\|_{C^1([0,T])}+\|w,\,z\|_{L^p(\dt)}).\lbl{2.15}
\end{align}

Take the difference of the equations for $h_1,h_2$ results in
\[ h'(t)=\mu\kk(  \frac{1}{h_2}z_{2,y}(t,1)-\frac{1}{h_1}z_{1,y}(t,1) \rr),\ \ 0\le t\le T;\ h(0)=0. \]
Therefore
\begin{align*}
 \|h'(t)\|_{C^{\frac{\ap}2}([0,T])}&=\mu\|h_2^{-1}z_{2,y}(t,1)-h_1^{-1}z_{1,y}(t,1)\|_{C^{\frac{\ap}2}([0,T])} \\
 &\le \mu\|h_1^{-1}z_y\|_{C^{\frac{\ap}2,0}(\dt)}+\mu\|z_{2,y}(h_1^{-1}-h_2^{-1})\|_{C^{\frac{\ap}2,0}(\dt)} \\
 &\le C_{13}(\|h\|_{C^1([0,T])}+\|z\|_{L^p(\dt)}),
\end{align*}
where $C_{13}$ is independent of $T$. Thanks to $h(0)=h'(0)=0$, we have
 \[\|h\|_{C^1([0,T])}\le2T^{\frac{\ap}2}\|h'(t)\|_{C^{\frac{\ap}2}([0,T])}
 \le 2C_{13}T^{\frac{\ap}2}(\|h\|_{C^1([0,T])}+\|z\|_{L^p(\dt)}). \]
Combined this with \eqref{2.13} and \eqref{2.15} allows us to derive
\begin{align*}
\|w,\,z\|_{C^{\frac{\ap}2,0}(\dt)}+\|h\|_{C^1([0,T])}
\le& C_{14}T^{\frac{\ap}2}\big(\|h\|_{C^1([0,T])}+\|w,\,z\|_{L^p(\dt)}\big)\\
\le& C_{15}T^{\frac{\ap}2}\big(\|h\|_{C^1([0,T])}+\|w,\,z\|_{C^{\frac{\ap}2,0}(\dt)}\big),
\end{align*}
where $C_{14},C_{15}$ are independent of $T$. Thus, when $0<T\ll 1$, we have $w=z=0$ and $h=0$, i.e., $w_1=w_2,z_1=z_2$ and $h_1=h_2$.

Recalling the transformation at the beginning of the proof, we thus conclude that, for $T>0$ small enough, the problem \eqref{1.1} admits a unique classical solution $(u,v,h)$, i.e., \eqref{2.1} holds. Moreover, from the proof we also see that $u,v>0$ in $(0,T]\times[0,h(t))$, $h'(t)>0$ in $(0,T]$ and \eqref{2.2} holds.

{\it Step 3: The bounds}. Let $0<\tau<\infty$ and $(u,v,h)\in [C^{1+\frac{\alpha}2,2+\alpha}(D_{\tau})]^2\times C^{1+\frac{1+\alpha}2}([0,\tau])$ be the unique solution of \eqref{1.1}. It follows from the maximum principle that $u,v>0$ in $(0,\tau]\times[0,h(t))$ and $h'(t)>0$ in $(0,\tau]$. By \eqref{1.1} and the nonnegativity of $u$ and $v$, we can see that $v$ satisfies
 \bess
 \left\{\begin{array}{lll}
 v_t-dv_{xx}=g(u,v)\le v(q-v),\ &0<t\leq \tau, \ \ 0<x<h(t),\\[1mm]
 v_x(t,0)=v(t,h(t))=0,\ \ \ &0\leq t\le \tau,\\[1mm]
 v(0,x)=v_0(x),\ \ &0\leq x\leq h_0.
 \end{array}\right.
 \eess
It follows that
\bes
  v\leq \max\{q,\|v_0\|_{C([0,h_0])}\}:=M_2\ \ \ {\rm on}\ \ [0, \tau]\times[0, h(t)]. \label{2.8}
 \ees

Next, we prove that there exists $M_3>0$ such that $h'(t)\leq M_3$ for $t\in (0,\tau]$. Let $K$ be a positive constant and
 \[Q_K:=\{0<t<\tau, \ h(t)-1/K<x<h(t)\}.\]
Clearly, $Q_K$ is well defined if we set $K\ge1/{h_0}$. Introducing an auxiliary function
\[ w(t,x)= M_2[2K(h(t)-x)-K^2(h(t)-x)^2].\]
The number $K$ will be chosen to ensure that $w\geq v$ in $Q_K$. By straightforward calculation, we obtain
\[ w_t= 2M_2K h'(t)[1-K(h(t)-x)]\geq0,\ \ w_{xx}=-2M_2K^2. \]
It follows from \eqref{2.8} that $g(u,v)\leq qM_2$. Hence,
\[ w_t-dw_{xx}\geq 2dM_2K^2\geq qM_2\ge g(u,v)\ \ {\rm in}\ \, Q_K\]
provided $K\geq\sqrt{\frac{q}{2d}}$. It is easy to see that
\[w(t,h(t)-{1}/K)=M_2\geq v(t,h(t)-1/K),\ \ w(t,h(t))=0=v(t,h(t)).\]
Note that
\[v_0(x)=\int_x^{h_0} v_0'(y){\rm d}y\leq (h_0-x)\|v_0'\|_{C([0,h_0])},\ \ \forall x\in[h_0-1/K,h_0]\]
and
\[w(0,x)=M_2[2K(h_0-x)-K^2(h_0-x)^2]\geq M_2K(h_0-x),\ \ \forall x\in[h_0-1/K,h_0].\]
Therefore, if $M_2K\geq \|v_0'\|_{C([0,h_0])}$, then
 \[ v_0(x)\leq (h_0-x)\|v_0'\|_{C([0,h_0])}\leq w(0,x),\ \ \forall x\in[h_0-1/K,h_0].\]

Let
 \[K=\max\left\{\frac{1}{h_0},\ \sqrt{\frac{q}{2d}}, \
 \frac{\|v_0'\|_{C([0,h_0])}}{M_2} \right\}.\]
Applying the maximum principle to $w-v$ over $Q_K$ yields that $v(t,x)\leq w(t,x)$
for $(t,x)\in Q_K$. It follows that $v_x(t,h(t))\geq w_x(t,h(t))=-2M_2K$. Hence,
 \[h'(t)=-\mu v_x(t,h(t))\leq 2\mu M_2K:=M_3 \ \ \ {\rm on} \ \ [0,\tau].\]

Finally, we show that
 \[u\leq\max\kk\{ u_m, \ \|u_0\|_\infty, \ \frac{(b-am)M_2}a-c \rr\}:=M_1\ \ \ {\rm on}\ \ [0, \tau]\times[0, h(t)]. \]
Assume on the contrary that there is $(t_0,x_0)\in(0,T]\times[0,h(t))$ fulfilling $u(t_0,x_0)=\dd\max_{[0,T]\times[0,h(t)]}u>M_1$. By \eqref{1.1}, $u$ satisfies
 \bes
 \left\{\begin{array}{lll}
 u_t-u_{xx}+u_x\chi(u)v_x+u\chi'(u)u_xv_x+u\chi(u)v_{xx}\\[1.5mm]
 \qquad=u\dd\frac{(b-am)v-ac-au}{c+u+mv}, \ \qquad 0<t\leq T,\ \ 0<x<h(t),\\[1.5mm]
 u_x(t,0)=u(t,h(t))=0,\quad\qquad \qquad  0\leq t\leq T,\\[1mm]
  u(0,x)=u_0(x),\ \ v(0,x)=v_0(x),\quad 0\leq x\leq h_0.
 \end{array}\right.\label{212a}
 \ees
We first assert that $x_0\notin(0,h(t))$. Assume on the contrary that $x_0\in(0,h(t))$. Then $u_t\ge0, u_x=0,u_{xx}\le0$ at $(t_0,x_0)$ and $\chi(u(t_0,x_0))=\chi'(u(t_0,x_0))=0$ because of $u(t_0,x_0)>M_1\ge u_m$. Thus
 \[ u_t-u_{xx}+u_x\chi(u)v_x+u\chi'(u)u_xv_x+u\chi(u)v_{xx}\ge 0\ \ {\rm at}\ (t_0,x_0). \]
And so $(b-am)v-ac-au\ge 0$ at $(t_0,x_0)$. This yields
  \[aM_1<au(t_0,x_0)\le (b-am)v(t_0,x_0)-ac, \]
which is a contradiction and hence $x_0\notin(0,h(t))$. Therefore, $x_0=0$. However, by Hopf boundary lemma we see that $u_x(t,0)<0$, which is impossible. The proof is complete.
\end{proof}

The next theorem guarantees that the solution of \eqref{1.1} exist globally.
 %Its proof is an adaption of that of \cite[Section 3]{YM}.

\begin{theo}\label{t2.2}\, Assume that $\chi(u)$, and $(u_0,v_0)$ satisfy the conditions \eqref{1.2} and \eqref{1.3}, respectively. Then the solution $(u,v,h)$ of problem \eqref{1.1} exists and is unique for all $t>0$. Moreover, the unique global solution $(u,v,h)$ of \eqref{1.1} satisfies
  \bess
(u,v,h)\in C^{1+\frac{\alpha}2,2+\alpha}(D_\infty)\times  C^{1+\frac{\alpha}2,2+\alpha}(D_\infty)\times C^{1+\frac{\alpha}2}([0,\infty)),
 \eess
where $D_\infty:= [0,\infty)\times[0,h(t)]$.
\end{theo}

\begin{proof}Let $T_0$ be the maximal existence time. On the contrary we assume that $T_0<\infty$.
The same as the proof of Theorem \ref{t2.1}, we set
\[y=x/h(t),\ \ w(t,y)=u(t,h(t)y),\ \ z(t,y)=v(t,h(t)y).\]
Then, for any $0<T<T_0$, $(w,z,h)$ satisfies
 \bes
 \left\{\begin{array}{lll}
 w_t-\zeta(t)w_{yy}-\xi(t)yw_y+\zeta(t)\eta'(w)z_y w_y\\[1mm]
 \quad=-\zeta(t)\eta(w)z_{yy}+f(w,z),\ \ &0<t\leq T, \ \ 0<y<1,\\[1mm]
 w_y(t,0)=w(t,1)=0,\ \ \ &0\leq t\leq T,\\[1mm]
 w(0,y)=u_0(h_0y),\ \ \ &0\leq y\leq 1,
 \end{array}\right.\label{2.10}
\ees
and
 \bes
 \left\{\begin{array}{lll}
 z_t-d\zeta(t)z_{yy}-\xi(t)yz_y=g(w,z),\ &0<t\leq T, \ \ 0<y<1,\\[1mm]
 z_y(t,0)=z(t,1)=0,\ \ \ &0\leq t\le T,\\[1mm]
 z(0,y)=v_0(h_0y),\ \ \ &0\leq y\leq 1,
 \end{array}\right.\label{2.11}
\ees
and
\bes
h'(t)=-\mu \frac{1}{h(t)}z_y(t,1),\ \ 0< t\leq T;\ \ h(0)=h_0.\label{2.12}
\ees
According to Theorem \ref{t2.1} (ii),
\[  0<\zeta(t)\leq \frac{1}{h_0^2},\ \ 0\leq\xi(t)\leq \frac{M_3}{h_0},\ \
\|g(w,z)\|_{L^\infty(\dt)}\leq C, \]
where $C$ is independent of $T$. Apply the parabolic $L^p$ theory to \eqref{2.11} yields that
\bes
\|z\|_{W_p^{1,2}(\dt)}\le C(T)(\|g(w,z)\|_{L^p(\dt)}+\|v_0\|_{L^p((0,h_0))})\leq C(T),\label{213a}
\ees
Notice that the respective constant on the most right-hand side of \eqref{213a} depends on $T$ only through an upper bound for $T$. Therefore, we have
\[ \|z\|_{W_p^{1,2}(\dt)}\leq C(T_0). \]
Then the embedding theorem gives $\|z_y\|_{L^\infty(\dt)}
\leq C(T_0)$. Therefore
 \[ \|\zeta(t)\eta'(w)z_y\|_{L^\infty(\dt)}\leq C(T_0),\ \
 \| -\zeta(t)\eta(w)z_{yy}+f(w,z) \|_{L^p(\dt)}\leq C(T_0). \]
Now we can apply the parabolic $L^p$ theory to \eqref{2.10} to derive
$$\|w\|_{W_p^{1,2}(\dt)}\leq C(T)(\| -\zeta(t)\eta(w)z_{yy}+f(w,z) \|_{L^p(\dt)}+\|u_0\|_{W^2_p((0,h_0))})\le C(T_0).$$
For any fixed $\alpha\in(0,1)$, we take $p$ large enough such that $1-3/p>\alpha$.
Then $\|w,z\|_{C^{(1+\alpha)/2,1+\alpha}(\dt)}\leq C(T_0)$ by the embedding theorem. Hence,
$\|g(w,z)\|_{C^{\alpha/2,\alpha}(\dt)}\leq C(T_0)$. Moreover, from \eqref{2.12}, we have
$\|h'\|_{C^{\alpha/2}([0,T])}\leq C(T_0)$, which implies $\|\zeta,\,\xi\|_ {C^{\alpha/2}([0,T])}
\leq C(T_0)$. Using the parabolic Schauder theory for \eqref{2.11} we find
 \[ \|z\|_{C^{1+\alpha/2,2+\alpha}(\dt)}\le C(T)(\|g(w,z)\|_{C^{\alpha/2,\alpha}(\dt)}+\|v_0\|_{C^{2+\ap}([0,h_0])})\le C(T_0). \]
It follows that
 \[ \|f(w,z)-\zeta(t)\eta(w)z_{yy}\|_{C^{\alpha/2,\alpha}(\dt)}\leq C(T_0). \]
Now, we can apply the parabolic Schauder theory to \eqref{2.10} to get
 \[ \|w\|_{C^{1+\alpha/2,2+\alpha}(\dt)}\le C(t)(\|f(w,z)-\zeta(t)\eta(w)z_{yy}\|_{C^{\alpha/2,\alpha}(\dt)}+\|u_0\|_{C^{2+\ap}([0,h_0])})\le C(T_0). \]

Recalling the transformation, we conclude that
\bes
\|u,v\|_{C^{1+\alpha/2,2+\alpha}([0,T]\times[0,h(t)])}\leq C(T_0)\ \ \ {\rm for}\ 0<T<T_0,
\lbl{2.21}\ees
and hence
\bes
 h(t)>h_0, \ \  \|h\|_{C^{1+\frac{1+\ap}2}([0,T])}\le C(T_0)\ \ \ {\rm for}\ 0<T<T_0.
 \lbl{2.22}\ees
Take $t_0\in (0,T_0)$ as a new initial time. Based on the proof of Theorem \ref{t2.1}, we can choose
a positive constant $\tau$, which depends only on the upper bound of $\|(u,v)(t_0,x)\|_{C^{2+\alpha}([0,h(t_0)])}$, $h(t_0)$ and $h'(t_0)$, and thus depends only on $T_0$ by \eqref{2.21} and \eqref{2.22}, such that the solution $(u,v,h)$ can be extended to $[0,t_0+\tau]\times[0,h(t)]$. If we
set $t_0<T_0$ such that $t_0+\tau>T_0$, then a contradiction is obtained.
\end{proof}

\section{Regularity and estimates}

Theorem \ref{t2.1} asserts that $h(t)$ is monotonic increasing. Thereby,
$\dd\lim_{t\to\infty}h(t)=h_\infty$ where $h_\infty\in(0,\infty]$. This section provides an
estimate for $(u,v,h)$ which plays an essential role in the sequel.

\begin{theo}\label{t3.1}\, Let $(u,v,h)$ be the unique global solution of \eqref{1.1} and
$D_\infty:=[0,\infty)\times[0,h(t)]$. Then there is a positive constant $C$ such that
\bes
\|u\|_{C^{1+\frac\alpha2,2+\alpha}(D_\infty)}\leq C, \ \
\|v\|_{C^{1+\frac\alpha2,2+\alpha}(D_\infty)}\leq C. \label{3.1}
\ees
The second estimate in \eqref{3.1} implies
\bes
\|h'\|_{C^{\frac{1+\alpha}2}(\bar {\mathbb{R}}_+)}\leq C. \label{3.2}
\ees
\end{theo}

\begin{proof} The idea of this proof comes from \cite[Theorem 2.1]{Wjde15},
\cite[Theorem 2.1]{Wjfa16} and \cite[Theorem 2.2]{WZppd}. In the following arguments we always take $p>3/(1-\alpha)$.

By Theorems \ref{t2.1} and \ref{t2.2}, the unique global solution $(u,v,h)$ of \eqref{1.1} satisfies
\bes
(u,v,h)\in C^{1+\frac{\alpha}2,2+\alpha}(D_\infty)\times C^{1+\frac{\alpha}2,2+\alpha}(D_\infty)\times C^{1+\frac{\alpha}2}({\mathbb{R}}_+). \label{3.3}
\ees
The same as the proof of Theorem \ref{t2.2} we define
\[y=x/h(t),\ \ w(t,y)=u(t,h(t)y),\ \ z(t,y)=v(t,h(t)y).\]
Then $(w,z,h)$ satisfies \eqref{2.10},  \eqref{2.11} and \eqref{2.12} for all $T>0$. Denote $Q_1:=[0,3]\times[0,1]$. Using Theorem \ref{t2.1} we infer that
\[  0<\zeta(t)\leq \frac{1}{h_0^2}, \ \ 0\leq\xi(t)\leq \frac{M_3}{h_0}, \ \
\|g(w,z)\|_{L^\infty([0,3]\times[0,1])}\leq C. \]
Applying the parabolic $L^p$ theory to \eqref{2.11} firstly and using the embedding theorem secondly, we have
\bes
\|z\|_{W_p^{1,2}(Q_1)}+\|z\|_{C^{\frac{1+\alpha}2,1+\alpha}(Q_1)}\leq C. \label{3.6}
\ees
This yields
\[ \|\zeta(t)\eta'(w)z_y\|_{L^\infty(Q_1)}+\|-\zeta(t)\eta(w)z_{yy}+f(w,z)\|_{L^p(Q_1)}\leq C. \]
Again, we apply the parabolic $L^p$ estimate to \eqref{2.10} and embedding theorem to get
\bes
\|w\|_{C^{\frac{1+\alpha}2,1+\alpha}(Q_1)}\leq C. \label{3.6+}
\ees

{\it Case 1:} $h_\infty<\infty$. For the integral $i\geq 0$, we denote
 \bess
 &z^i(t,y)=z(t+i,y), \ \ \ w^i(t,y)=w(t+i,y),&\\[1mm]
 &h^i(t)=h(t+i), \ \ \zeta^i(t)=\zeta(t+i), \ \ \xi^i(t)=\xi(t+i).\eess
 Then $w^i(t,y)$ and $z^i(t,y)$ satisfies
\bes
 \left\{\begin{array}{lll}
 w^i_t-\zeta^i(t)w^i_{yy}-\xi^i(t)yw^i_y+\zeta^i(t)\eta'(w^i)z^i_y w^i_y\\[1mm]
 \qquad=-\zeta^i(t)\eta(w^i)z^i_{yy}+f^i,\ \ &0<t\leq 4, \ \ 0<y<1,\\[1mm]
 w^i_y(t,0)=w^i(t,1)=0,\ \ \ &0\leq t\leq 4,\\[1mm]
 w^i(0,y)=u(i,h(i)y),\ \ \ &0\leq y\leq 1,
 \end{array}\right.\label{3.7-}
\ees
 and
\bes
 \left\{\begin{array}{lll}
 z^i_t-d\zeta^i(t)z^i_{yy}-\xi^i(t)yz^i_y=g^i,\ \ &0<t\leq 4, \ \ 0<y<1,\\[1mm]
 z^i_y(t,0)=z^i(t,1)=0,\ \ \ &0\leq t\leq 4,\\[1mm]
 z^i(0,y)=v(i,h(i)y),\ \ \ &0\leq y\leq 1,
 \end{array}\right.\label{3.7}
\ees
where $f^i=f(w^i,z^i)$, $g^i=g(w^i,z^i)$. In view of \eqref{2.5}-\eqref{2.7}
we deduce that functions $w^i,z^i,\zeta^i,\xi^i,f^i$ and $g^i$ are uniformly bounded and the modulus of continuity
 \[ \rho^i:= \underset{0\leq t,s\leq 4,|t-s|\leq \delta}\max|\zeta^i(t)-\zeta^i(s)|\leq
 \frac{2M_3}{h_0^3}\delta \to 0\ \ {\rm as}\ \ \delta\to 0 \]
uniformly on $i$. Similarly to the arguments of \cite[Theorem 2.1]{Wjfa16}
we can obtain
 \bes
 \|z^i\|_{W^{1,2}_p([1,4]\times[0,1])}+\|z^i\|_{C^{\frac{1+\alpha}2,1+\alpha}([1,4]\times[0,1])}\leq C, \ \ \forall \ i\ge 0, \label{3.8}
  \ees
and
 \bess
 \|v\|_{C^{\frac{1+\alpha}2,1+\alpha}(D_\infty)}\leq C.
 \eess
The last estimate implies
 \bes
\|h'\|_{C^{\frac{\alpha}2}(\bar {\mathbb{R}}_+)}\leq C. \lbl{3.10}
 \ees

By \eqref{3.8}, we have
 \[\|\zeta^i\eta'(w^i)z^i_y\|_{L^\yy([1,4]\times[0,1])}+ \|\zeta^i\eta(w^i)z^i_{yy}\|_{L^p([1,4]\times[0,1])}\le C,\ \ \forall \ i\ge 0. \]
Hence we can apply the interior $L^p$ estimate (\cite[Theorems 7.30 and 7.35]{Lie})
to \eqref{3.7-} to deduce that there exists a constant $C>0$ such that
$\|w^i\|_{W_p^{1,2}([2,4]\times[\frac{1}2,1])}\leq C$ and $\|w^i\|_{W_p^{1,2}([2,4]
\times[0,\frac{1}2])}\leq C$ for integer $i\ge 0$.
Thus we have
 \[\|w\|_{C^{\frac{1+\alpha}2,1+\alpha}(B_i)}\leq C,\ \ \forall \ i\ge 0, \]
where $B_i=[i+2,i+4]\times[0,1]$. Since these rectangles $B_i$ overlap and $C$ is independent
of $i$, it follows that $\|w\|_{C^{\frac{1+\alpha}2,1+\alpha}([2,\infty)\times[0,1])}\leq C$.
This in conjunction with \eqref{3.6+} yields
 \bes
 \|w\|_{C^{\frac{1+\alpha}2,1+\alpha}([0,\infty)
\times[0,1])}\leq C.\lbl{3.10a}\ees

Thanks to \eqref{3.8}, \eqref{3.10} and \eqref{3.10a}, we see that functions $\zeta^i,\xi^i$ and $g^i$ are H\"{o}lder continuous uniformly in $i$, i.e., there exists a
positive constant $C$ independent of $i$ such that
 \[\|\zeta^i,\xi^i\|_{C^{\alpha/2}([0,4])}\le C, \ \ \
 \|g^i\|_{C^{\alpha/2,\,\alpha}([0,4]\times[0,1])}\le C.\]
Using the global Schauder estimate for $i=0$ and the interior Schauder estimate for $i\ge 1$ (\cite{L-S-Y1968}), we obtain
\bes
 \|z\|_{C^{1+\frac{\alpha}2,2+\alpha}([0,4]\times[0,1])}&\leq& C\big(\|g\|_{C^{\frac{\alpha}2,\alpha}([0,4]\times[0,1])}+\|v_0\|_{C^{2+\alpha}([0,h_0])}\big)\leq C,\nonumber\\[1.5mm]
 \|z^i\|_{C^{1+\frac{\alpha}2,2+\alpha}([1,4]\times[0,1])}&\leq& C\big(\|g^i\|_{C^{\frac{\alpha}2,\alpha}([0,4]\times[0,1])}+\|z^i\|_\infty\big)\nonumber\\[1mm]
  &\leq& C\big(\|g^i\|_{C^{\frac{\alpha}2,\alpha}([0,4]\times[0,1])}+M_2\big) \leq C,\ \ \forall \ i\ge 0,  \lbl{3.11a}\ees
which entails that
 \[\|z\|_{C^{1+\frac{\alpha}2,2+\alpha}([0,\infty)\times[0,1])}\leq C. \]

The estimates \eqref{3.11a} and \eqref{3.8} show that $\zeta^i\eta'(w^i)z^i_y$, $\zeta^i\eta(w^i)z^i_{yy}$ and $f^i$ are H\"{o}lder continuous on $[1,4]\times[0,1]$ and uniformly with respect to $i\ge 1$, i.e., there exists a constant $C>0$ such that
 \[\|\zeta^i\eta'(w^i)z^i_y, \ \zeta^i\eta(w^i)z^i_{yy}, \ f^i\|_{C^{\alpha/2,\,\alpha}([1,4]\times[0,1])}\le C,\ \ \forall \ i\ge 1.\]
Similar to the above, we can apply the Schauder estimate to $w$ and $w^i$ and obtain
 \[\|w\|_{C^{1+\frac{\alpha}2,2+\alpha}([0,4]\times[0,1])} \leq C, \ \ \
 \|w^i\|_{C^{1+\frac{\alpha}2,2+\alpha}([2,4]\times[0,1])} \leq C\ \ {\rm for}\ i\ge1.\]
Therefore,
 \[\|w\|_{C^{1+\frac{\alpha}2,2+\alpha}([0,\infty)\times[0,1])}\leq C. \]
Take advantage of \eqref{3.10} and the relations:
 \bes
 u_t=w_t-\frac{h'}hyw_y, \ \ u_{xx}=\frac 1{h^2}w_{yy},\ \ v_t=z_t-\frac{h'}hyz_y, \ \ v_{xx}=\frac 1{h^2}z_{yy},
 \lbl{3.11}\ees
it is easily to derive the estimates \eqref{3.1}.

{\it Case 2:} $h_\infty=\infty$. For any fixed $i_1, i_2<\infty$, an argument similar to the one used in {\it Case 1} shows that
\bes
\|u\|_{C^{1+\frac\alpha2,2+\alpha}([0,i_1]\times[0,h(t)])}\leq C(i_1),\label{3.12}\\[1mm]
\|v\|_{C^{1+\frac\alpha2,2+\alpha}([0,i_2]\times[0,h(t)])}\leq C(i_2). \label{3.13}
\ees

For the integer $i\geq 0$, let
\[ u^i(t,x)=u(t+i,x),\ v^i(t,x)=v(t+i,x),\ h^i(t)=h(t+i).\]
Then $u^i, v^i$ satisfies
\bes
 \left\{\begin{array}{lll}
 u^i_t-u^i_{xx}+\eta'(u^i)v^i_xu^i_x=-\eta(u^i)v^i_{xx}+f(u^i,v^i),\ &0<t\leq7, \ \ 0<x<h^i(t),\\[1mm]
 u^i_x(t,0)=u^i(t,h^i(t))=0,\ \ \ &0\leq t\leq7,\\[1mm]
 u^i(0,x)=u(i,x),\ \ \ &0\leq x\leq h(i),
 \end{array}\right.\label{3.14}
\ees
and
 \bes
 \left\{\begin{array}{lll}
 v^i_t-dv^i_{xx}=g(u^i,v^i),\ \ &0<t\leq7, \ \ 0<x<h^i(t),\\[1mm]
 v^i_x(t,0)=v^i(t,h^i(t))=0,\ \ \ &0\leq t\leq7,\\[1mm]
 v^i(0,x)=v(i,x),\ \ \ &0\leq x\leq h(i).
 \end{array}\right.\label{3.15}
\ees
From Theorem \ref{t2.1}, $u^i,v^i,h^i$ are bounded uniformly on $i$, and
  \[ h^i(t)\leq h(i)+M_3t\leq h(i)+7M_3 \]
for all $0\leq t\leq 7$. Set $\sigma=7M_3$ and $\sigma_n=\sigma+n$ from now on. Then, it is easy to see that $h(i)\geq h^i(t)-\sigma$ for all $0\leq t\leq 7$. As $h_\infty=\infty$, there is an $i'\geq 0$ fulfilling $h(i')>\sigma+9$.

With \eqref{3.12} and \eqref{3.13} at hand, we only need to study $u^i, v^i$ with $i\geq i'$. Clearly, $h(i)>\sigma+9$ for all $i\geq i'$. Let us choose $p\gg 1$. For any integer $0\leq l \leq h(i)-9$, we can apply the interior $L^p$ estimate (\cite[Theorems 7.30 and 7.35]{Lie}) to the problem \eqref{3.15} and deduce that
\bes
 \|v^i\|_{W_p^{1,2}((1,7)\times(l,l+8))}+\|v^i\|_{C^{\frac{1+\ap}2,1+\ap}([1,7]\times[l,l+8])}\leq C,\ \ \forall\ 0\le l\le h(i)-9,  \label{3.16}
\ees
where $C>0$ is independent of $l$ and $i$. The same as the proof of \cite[Theorem 2.2]{WZppd}
we can get
 \bess
\|v\|_{C^{\frac{1+\alpha}2,1+\alpha}([i'+1,\infty]\times[0,h(t)-\sigma_2])}\leq C.
\eess
Set $Q_l'=(1,7)\times(l,l+8)$ and $Q_l''=(2,6)\times(l+1,l+7)$.
The estimate \eqref{3.16} implies
 \[\|v^i_x\|_{L^\infty(Q')}\le C, \ \ \ \|v^i_{xx}\|_{L^p(Q')}\le C,\ \ \forall \ i\ge i'.\]
When $l=0$, an application of interior $L^p$ estimate (\cite[Theorems 7.35]{Lie}) yields that
\bes
 \|u^i\|_{W_p^{1,2}((2,6)\times(0,7))}+\|u^i\|_{C^{\frac{1+\ap}2,1+\ap}([2,6]\times[0,7])}
 \leq C,\ \ \forall \ i\ge i'. \label{3.21a}
 \ees
When $l\ge1$, since $Q_l''\subset\subset Q_l'$, we are able to apply the interior $L^p$ estimate (\cite[Theorem 7.22]{Lie}) to the problem \eqref{3.14} to obtain
 \bes
\|u^i\|_{W_p^{1,2}((2,6)\times(l+1,l+7))}+\|u^i\|_{C^{\frac{1+\ap}2,1+\ap}([2,6]\times[l+1,l+7])}\leq C \label{3.22a}
 \ees
for some $C>0$ independent of $i$ and $l$. Use the arguments in the proof of \cite[Theorem 2.2]{WZppd}, it can be deduced that
\bess
\|u\|_{C^{\frac{1+\alpha}2,1+\alpha}([i'+2,\,\infty]\times[0,\,h(t)-\sigma_3])}\leq C.
\eess

Thanks to \eqref{3.16}, \eqref{3.21a} and \eqref{3.22a}, we see that functions $g(u^i,v^i)$ are H\"{o}lder continuous uniformly in $i$, i.e., there exists a
positive constant $C$ independent of $i$ such that
\[\|g(u^i,v^i)\|_{C^{\alpha/2,\,\alpha}([2,6]\times[0,7])}\le C,\]
and
 \[ \|g(u^i,v^i)\|_{C^{\alpha/2,\,\alpha}([2,6]\times[l+1,l+7])}\le C, \ \ \ l\ge1.\]
Using the interior Schauder estimate for $i\ge i'$ (\cite{L-S-Y1968}), we obtain
\bes
 &\|v^i\|_{C^{1+\frac{\alpha}2,2+\alpha}([3,6]\times[0,6])} \leq C(\|g(u^i,v^i)\|_{C^{\frac{\alpha}2,\alpha}([2,6]\times[0,7])}+M_2)\leq C,&\label{3.23a}\\[1.5mm]
 &\|v^i\|_{C^{1+\frac{\alpha}2,2+\alpha}([3,6]\times[l+2,l+6])} \leq C(\|g(u^i,v^i)\|_{C^{\frac{\alpha}2,\alpha}([2,6]\times[l+1,l+7])}+M_2) \leq C,\ \ \ l\ge1,&\label{3.24a}  \ees
which entails that
 \bes
 \|v\|_{C^{1+\frac{\alpha}2,2+\alpha}([i'+3,\infty)\times[0,h(t)-\sigma_4])}\leq C.\label{3.17a}
 \ees

The estimates \eqref{3.22a}, \eqref{3.23a} and \eqref{3.24a} show that $\eta'(u^i)v^i_x,\ -\eta(u^i)v^i_{xx}+f(u^i,v^i)$ are H\"{o}lder continuous on $[3,6]\times[0,6]\cup[3,6]\times[l+2,l+6]$ and uniformly with respect to $i\ge i',l\ge1$, i.e., there exists a constant $C>0$ such that
 \[\|\eta'(u^i)v^i_x, \ f(u^i,v^i)-\eta(u^i)v^i_{xx}\|_{C^{\alpha/2,\,\alpha}([3,6]\times[0,6])}\le C,\ \ \forall \ i\ge i',\]
 and
 \[\|\eta'(u^i)v^i_x, \ f(u^i,v^i)-\eta(u^i)v^i_{xx}\|_{C^{\alpha/2,\,\alpha}([3,6]\times[l+2,l+6])}\le C,\ \ \forall \ i\ge i',\ l\ge 1.\]
Similar to the above, we can apply the interior Schauder estimate to $u^i$ and obtain
 \[\|u^i\|_{C^{1+\frac{\alpha}2,2+\alpha}([4,6]\times[0,5])} \leq C, \ \ \
 \|u^i\|_{C^{1+\frac{\alpha}2,2+\alpha}([4,6]\times[l+3,l+5])} \leq C,\ \ \forall \ i\ge i', \ l\ge1.\]
Therefore,
\bes
\|u\|_{C^{1+\frac{\alpha}2,2+\alpha}([i'+4,\infty)\times[0,h(t)-\sigma_5])}\leq C. \label{3.18a}
\ees

Next, we shall show
\bess
\|u\|_{C^{1+\frac{\alpha}2,2+\alpha}([i'+4,\infty]\times[h(t)-\sigma_5,h(t)])}\leq C, \\
\|v\|_{C^{\frac{1+\alpha}2,1+\alpha}([i'+3,\infty]\times[h(t)-\sigma_4,h(t)])}\leq C.
\eess
Using the transformation
\[ y=h(t)-x,\ \phi(t,y)=u(t,h(t)-y),\ \psi(t,y)=v(t,h(t)-y) \]
we deduce that $\phi,\psi$ satisfies
\bess
 \left\{\begin{array}{lll}
 \phi_t-\phi_{yy}+h'(t)\phi_y+\eta'(\phi)\phi_y\psi_y=-\eta(\phi)\psi_{yy}+f(\phi,\psi),\ \ &0<t\le 7, \ \ 0<y<h(t),\\[1mm]
 \phi(t,0)=\phi_y(t,h(t))=0,\ \ \ &0\le t\le 7,\\[1mm]
 \phi(0,y)=u_0(h_0-y),\ \ \ &0\leq y\leq h_0,
 \end{array}\right.
\eess
and
 \bess
 \left\{\begin{array}{lll}
 \psi_t-d\psi_{yy}+h'(t)\psi_y=g(\phi,\psi),\ \ &0<t\le 7, \ \ 0<y<h(t),\\[1mm]
 \psi(t,0)=\psi_y(t,h(t))=0,\ \ \ &0\le t\le 7,\\[1mm]
 \psi(0,y)=v_0(h_0-y),\ \ \ &0\leq y\leq h_0.
 \end{array}\right.
\eess
For the integer $i\geq i'$, we let
\[\phi^i(t,y)=\phi(t+i,y),\ \psi^i(t,y)=\psi(t+i,y),\ h^i(t)=h(t+i).\]
Then $\phi^i, \psi^i$ satisfies
\bes
 \left\{\begin{array}{lll}
 \phi^i_t-\phi^i_{yy}+[(h^i(t))'+\eta'(\phi^i)\psi^i_y]\phi^i_y\\[1mm]
 \qquad=-\eta(\phi^i)\psi^i_{yy}+f(\phi^i,\psi^i),\ \ &0<t\leq 7, \ \ 0<y<h^i(t),\\[1mm]
 \phi^i(t,0)=\phi^i_y(t,h^i(t))=0,\ \ \ &0\leq t\leq 7,\\[1mm]
 \phi^i(0,y)=\phi(i,y)=u(i,h(i)-y),\ \ \ &0\leq y\leq h(i),
 \end{array}\right.\label{3.29a}
\ees
and
 \bes
 \left\{\begin{array}{lll}
 \psi^i_t-d\psi^i_{yy}+(h^i(t))'\psi^i_y=g(\phi^i,\psi^i),\ \ &0<t\leq 7, \ \ 0<y<h^i(t),\\[1mm]
 \psi^i(t,0)=\psi^i_y(t,h^i(t))=0,\ \ \ &0\leq t\leq 7,\\[1mm]
 \psi^i(0,y)=\psi(i,y)=v(i,h(i)-y),\ \ \ &0\leq y\leq h(i).
 \end{array}\right.\label{3.30a}
\ees
In view of Theorem \ref{t2.1}, we know that $\phi^i, \psi^i, (h^i(t))', f(\phi^i,\psi^i), g(\phi^i,\psi^i)$ are bounded uniformly on $i$. Set $\Gamma_i=[i+1,i+7]\times[0,\sigma_8]$. Recall
\[ h(i)-\sigma_8\geq h(i')-\sigma_8>1,\ \ \forall \ i\geq i', \ 0\leq t\leq 7. \]
Applying the interior $L^p$ estimate (\cite[Theorems 7.30]{Lie}) to \eqref{3.30a} firstly and the embedding theorem secondly, we have
\bes
 \|\psi^i\|_{W_p^{1,2}(\Gamma_0)}+\|\psi^i\|_{C^{\frac{1+\alpha}2,1+\alpha}(\Gamma_0)}
 \leq C,\ \ \forall \ i\geq i'.  \label{3.31a}
\ees
Accordingly,
\[  \|\psi\|_{C^{\frac{1+\alpha}2,1+\alpha}(\Gamma_i)}\leq C. \]
Since these rectangles $\Gamma_i$ overlap and $C$ is independent of $i$, it follows that
\bes
\|\psi\|_{C^{\frac{1+\alpha}2,1+\alpha}([i'+1,\infty]\times[0,\sigma_8])}\leq C. \label{3.32a}
\ees

Set $\Sigma_i=[i+2,i+6]\times[0,\sigma_7]$. In view of \eqref{3.31a} and $\sigma_8=\sigma_7+1$, we can apply interior $L^p$ estimate (\cite[Theorems 7.30]{Lie}) to \eqref{3.29a} firstly and embedding theorem secondly to get
\bes
 \|\phi^i\|_{W_p^{1,2}(\Sigma_0)}+\|\phi^i\|_{C^{\frac{1+\alpha}2,1+\alpha}(\Sigma_0)}\leq C,\ \ \forall \ i\geq i'. \label{3.31b}
\ees
It follows that
\bess
\|\phi\|_{C^{\frac{1+\alpha}2,1+\alpha}([i'+2,\infty]\times[0,\sigma_7])}\leq C.
\eess

From \eqref{3.31a} and \eqref{3.31b} we have
\bess
\|(h^i(t))'\|_{C^{\frac{\ap}2}([2,6])}+ \|g(\phi^i,\psi^i)\|_{C^{\frac{\ap}2,\ap}([2,6]\times[0,\sigma_7])}\le C,\ \ \forall \ i\geq i'.
\eess
This allows us to apply the interior Schauder estimate to \eqref{3.30a} for $i\ge i'$ (\cite{L-S-Y1968}) to obtain
\bes
\|\psi^i\|_{C^{1+\frac{\ap}2,2+\ap}([3,6]\times[0,\sigma_6])}\le C(\|g(\phi^i,\psi^i)\|_{C^{\frac{\ap}2,\ap}([2,6]\times[0,\sigma_7])}+M_2)\le C,\ \ \forall \ i\geq i',\label{3.37a}
\ees
which implies that
\bes
\|\psi\|_{C^{1+\frac{\ap}2,2+\ap}([i'+3,\yy)\times[0,\sigma_6])}\le C.\label{3.38a}
\ees
The estimates \eqref{3.31b} and \eqref{3.37a} show that functions $(h^i(t))'+\eta'(\phi^i)\psi^i_y$ and $f(\phi^i,\psi^i)-\eta(\phi^i)\psi^i_{yy}$ are H\"{o}lder continuous on $[3,6]\times[0,\sigma_6]$ and uniformly with respect to $i\ge i'$, i.e., there exists a constant $C>0$ independent of $i$ such that
\bess
\|(h^i(t))'+\eta'(\phi^i)\psi^i_y\|_{C^{\frac{\ap}2,\ap}([3,6]\times[0,\sigma_6])}+ \|f(\phi^i,\psi^i)-\eta(\phi^i)\psi^i_{yy}\|_{C^{\frac{\ap}2,\ap}([3,6]\times[0,\sigma_6])}\le C.
\eess
Hence by the interior Schauder estimate (\cite{L-S-Y1968}) we get that, for all $i\ge i'$,
\bess
\|\phi^i\|_{C^{1+\frac{\ap}2,2+\ap}([4,6]\times[0,\sigma_5])}\le C\kk(\|f(\phi^i,\psi^i)-\eta(\phi^i)\psi^i_{yy}\|_{C^{\frac{\ap}2,\ap}([3,6]\times[0,\sigma_6])}+M_1\rr)\le C,
\eess
which yields
\bes
\|\phi\|_{C^{1+\frac{\ap}2,2+\ap}([i'+4,\yy)\times[0,\sigma_5])}\le C.\label{3.39a}
\ees

Notice that $0\leq y\leq \sigma_i$ is equivalent to $h(t)-\sigma_i\leq x\leq h(t)$, $i=5,6$. By means of $0<h'(t)\leq M_3$ and $v_x(t,x)=-\psi_y(t,y)$ ($u_x(t,x)=-\phi_y(t,y)$), It follows from \eqref{3.38a} and \eqref{3.39a} that
 \bes
 \|u\|_{C^{\frac{1+\alpha}2,1+\alpha}([i'+4,\,\yy]\times[h(t)-\sigma_5,\,h(t)])}\leq C,
 \label{3.40a}\\[1mm]
\|v\|_{C^{\frac{1+\alpha}2,1+\alpha}([i'+3,\,\yy]\times[h(t)-\sigma_6,\,h(t)])}\leq C.
 \nonumber\ees
Recalling that $\sigma_6=\sigma_4+2$, we have
\bes
\|v\|_{C^{\frac{1+\alpha}2,1+\alpha}([i'+3,\,\yy]\times[h(t)-\sigma_4,\,h(t)])}\leq C.\label{3.42a}
\ees

A combination of \eqref{3.18a} and \eqref{3.40a} implies
\bes
\|u\|_{C^{1+\frac\alpha2,2+\alpha}([i'+4,\infty]\times[0,h(t)])}\leq C, \label{3.43}
\ees
and a combination of \eqref{3.17a} and \eqref{3.42a} yields
\bes
\|v\|_{C^{1+\frac\alpha2,2+\alpha}([i'+3,\infty]\times[0,h(t)])}\leq C. \label{3.44}
\ees

Finally, using \eqref{3.12} with $i_1=i'+4$ and \eqref{3.43} we get the first estimate in \eqref{3.1}. Employing \eqref{3.13} with $i_2=i'+3$ and \eqref{3.44} we get the second estimate in \eqref{3.1}.
\end{proof}

\section{Long time behavior of $(u,v,h)$}
\subsection{Long time behavior of $(u,v)$ in the case of $h_\infty<\infty$}

\begin{theo}\label{t4.1}
Let $(u,v,h)$ be the unique global solution of the problem \eqref{1.1}. If $h_\infty<\infty$, then $\dd\lim_{t\to\infty}h'(t)=0$.
\end{theo}

\begin{proof}
In view of \eqref{2.7}, \eqref{3.2} and assumption $h_\infty<\infty$, it is easy to derive that $\dd\lim_{t\to\infty}h'(t)=0$.
\end{proof}

Next, we give a lemma which enables us to obtain the vanishing phenomenon.

\begin{lem}{\rm(\cite[Lemma 4.1]{WZppd})}\label{lemma 4.1}\, Let $d$, $C$, $\mu$ and $s_0$ be positive constants, $w\in W^{1,2}_p((0,T)\times(0,s(t)))$ for some $p>1$ and any $T>0$, and $w_x\in C([0,\infty)\times[0,s(t)])$, $s\in C^1([0,\infty))$. If $(w,s)$ satisfies
  \vspace{-.5mm}\bess\left\{\begin{array}{lll}
 w_t-dv_{xx}\geq -C w, &t>0,\ \ 0<x<s(t),\\[1mm]
 w\ge 0,\ \ \ &t>0, \ \ x=0,\\[1mm]
 w=0,\ \ s'(t)\geq-\mu w_x, \ &t>0,\ \ x=s(t),\\[1mm]
 w(0,x)=w_0(x)\ge,\,\not\equiv 0,\ \ &x\in (0,m_0),\\[1mm]
 s(0)=s_0,
 \end{array}\right.\vspace{-.5mm}\eess
and
  \vspace{-.5mm}
  \bess
  &\dd\lim_{t\to\infty} s(t)=s_\infty<\infty, \ \ \ \lim_{t\to\infty} s'(t)=0,&\\[1mm]
& \|w(t,\cdot)\|_{C^1[0,\,s(t)]}\leq M, \ \ \forall \ t>1&
 \eess
for some constant $M>0$.  Then
  \vspace{-.5mm}\bess
 \lim_{t\to\infty}\,\max_{0\leq x\leq s(t)}w(t,x)=0.
 \lbl{zz7.25}\vspace{-1mm}\eess
\end{lem}

By using of Lemma \ref{lemma 4.1} we have the following theorem.

\begin{theo}\label{t4.2}
Let $(u,v,h)$ be the unique global solution of \eqref{1.1}. If $h_\infty<\infty$, then,
\bes
\lim_{t\to\infty}\|v(t,\cdot)\|_{C^2([0,h(t)])}=0. \label{4.4}
\ees
\end{theo}

\begin{proof}
According to Theorem \ref{t2.1}, we find that
\[g(u,v)\ge -v\kk(v+\dd\frac{ru}{c+u+mv}\rr)\ge -v(M_2+rM_1/c).\]
Recall \eqref{2.1}, the second estimate in \eqref{3.1} and Theorem \ref{t4.1}. An application of Lemma \ref{lemma 4.1} yields
 \bess
\lim_{t\to\infty}\|v(t,\cdot)\|_{C([0,h(t)])}=0.
\eess
The second estimate in \eqref{3.1} implies
 \bess
\max_{t\geq 1}\|v(t,\cdot)\|_{C^{2+\alpha}([0,h(t)])}\leq C.
 \eess
And so $\dd\lim_{t\to\infty}\|z(t,\cdot)\|_{C([0,1])}=0$ and
 \bess
\max_{t\geq 1}\|z(t,\cdot)\|_{C^{2+\alpha}([0,1])}\leq C,
 \eess
where $z(t,y)=v(t,h(t)y)$. Since $C^{2+\alpha}([0,1])\hookrightarrow\hookrightarrow C^2([0,1])$, there exists $t_k\to\infty$ such that $z(t_k,y)\to z^*(y)$ in $C^2([0,1])$. The uniqueness of the limit shows that $z^*(y)\equiv 0$. And then $z(t,y)\to 0$ in $C^2([0,1])$ as $t\to\infty$. This combined with \eqref{3.11} allows us to deduce \eqref{4.4}.
\end{proof}

\begin{remark}
Theorem \ref{t4.2} asserts that if the prey can not spread into the whole space, then it will die out.
\end{remark}

With the help of Theorem \ref{t4.2}, we show that the predator will die out in the case of $h_\infty<\infty$.

\begin{theo}\label{theorem 4.4}
Let $(u,v,h)$ be the unique global solution of \eqref{1.1}. If $h_\infty<\infty$, then
\bess
 \lim_{t\to\infty}\|u(t,\cdot)\|_{C([0,h(t)])}=0.
\eess
\end{theo}
\begin{proof}
In the system \eqref{1.1} we have
 \bes
 \left\{\begin{array}{lll}
 u_t-u_{xx}+\eta'(u)v_x u_x=-v_{xx}\chi(u)u+\dd\frac{buv}{c+u+mv}-au, \ \ &t>0,\ \ 0<x<h(t),\\[1mm]
 u_x(t,0)=u(t,h(t))=0,\ \ \ &t\ge0,\\[1mm]
  u(0,x)=u_0(x),\ \ &0\leq x\leq h_0.
 \end{array}\right.\label{th1}
 \ees
In view of \eqref{1.2}, there exists $E>0$ such that
 \[ |\chi(u)|<E\ \ \ {\rm for\ all}\ u\in[0,\infty).  \]
In view of \eqref{4.4}, for any fixed $\ep>0$ fulfilling $(E+1)\ep=a/2$, one can find $T_1>0$ such that
\[ |v_{xx}(t,x)|\leq \varepsilon\ \ \ \forall\ t\geq T_1,\ x\in[0,h(t)]. \]
Moreover, again by \eqref{4.4}, one can find $T_2>0$ such that
\[ 0<\dd\frac{bv}{c+u+mv}\leq \ep\ \ \ \forall\ t\geq T_2. \]
Set $T:=\max\{T_1,T_2\}$. It is well known that
$$\bu(t)=M_1e^{(E\varepsilon+\varepsilon-a)(t-T)}\ {\rm for}\ t\geq T$$
is a solution of the following equation
\bess
 \left\{\begin{array}{lll}
 \bu_t=(\varepsilon E-a+\varepsilon)\bu, &t>T,\\[1mm]
 \bu(T)=M_1.
 \end{array}\right.
 \eess
Then, $\bu(t)$ is an upper solution to the problem \eqref{th1}. Hence by the comparson principle (Lemma A.1) we have
\[ 0<u(t,x)\leq \bu(t)\ \ \forall\ t\geq T,\ x\in[0,h(t)]. \]
Due to $E\varepsilon+\varepsilon-a=-a/2<0$, we get
$\dd\lim_{t\to\infty}\bu(t)=0$,
which yields
\[ 0\leq\dd\limsup_{t\to\infty} u(t,x)\leq \dd\lim_{t\to\infty}\bu(t)=0 \]
uniformly for $x$ in $[0,h(t)]$. This completes the proof.
\end{proof}
\begin{remark}
In the application of the comparison principle (Lemma A.1), we used the assumption that $\chi(u)\in C^2([0,\infty))$, which is a necessary assumption for our result.
\end{remark}

\subsection{Long time behavior of $(v,h)$ in the case of $h_\infty=\infty$}

\begin{theo}\label{t4.5}
If $h_\infty=\infty$, then $\dd\limsup_{t\to\infty}v(t,x)\leq q$ uniformly for $x$ in any bounded set of $[0,\infty)$, and there exists $k_1>0$ such that
$\limsup\limits_{t\to\infty}\frac{h(t)}{t}\leq k_1$.
\end{theo}

\begin{proof}
We observe that $(v,h)$ is an lower solution of the following problem
\bess
 \left\{\begin{array}{lll}
 \bv_t-d\bv_{xx}= \bv(q-\bv),\ \ &t>0, \ \ 0<x<\bh(t),\\[1mm]
 \bv_x(t,0)=\bv(t,\bh(t))=0,\ \ \ &t\geq0,\\[1mm]
 \bh'(t)=-\mu \bv_x(t,\bh(t)), \ \ \ &t\geq0,\\[1mm]
 \bv(0,x)=v_0(x),\ \ \ &0\leq x\leq \bh(0),\\[1mm]
 \bh(0)=h_0.
 \end{array}\right.
 \eess
By the comparison principle, $h(t)\le\bh(t)$, $v(t,x)\leq\bv(t,x)$, and so $\bh_\infty\geq h_\infty=\infty$. Using \cite[Theorems 3.4 and 4.2]{DLin},  $\dd\lim_{t\to\infty}\bv(t,x)=q$ uniformly for $x$ in any bounded set of $[0,\infty)$ and $\lim\limits_{t\to\infty}\frac{\bh(t)}{t}=k_1$ for $k_1>0$. Hence, our conclusion follows from the comparison principle (Lemma A.2).
\end{proof}

\section{Conditions for spreading and vanishing}

\begin{theo}\label{theorem a5.1}
Suppose that $h_\infty<\infty$. Let $(u,v,h)$ be the unique global solution of \eqref{1.1}. Then,
\bess
h_\infty\leq \frac{\pi}2\sqrt{{d}/{q}}.
\eess
\end{theo}

\begin{proof} As $h_\infty<\infty$, it follows from Theorems \ref{t4.2} and \ref{theorem 4.4} that
$\dd\lim_{t\to\infty}\|u,\,v\|_{C([0,h(t)])}=0$.
For any small $\ep\in(0,q)$, there exists $\tau_1\gg 1$ such that
\[ \frac{ru}{c+u}\leq \ep,\ \ \forall\ t\geq\tau_1,\ x\in[0,h_\infty], \]
We assume that $h_\infty> \frac{\pi}2\sqrt{d/(q-\ep)}$. Then, there exists $\tau_2\gg 1$ such that
\[ h(t)>\frac{\pi}2\sqrt{d/(q-\ep)},\ \ \forall\ t\geq\tau_2. \]
Set $\tau=\max\{\tau_1,\tau_2\}$. Similar to the proof of \cite[Lemma 3.1]{DLin}, we can construct a function $\ud v$ which satisfies
\bes
 \left\{\begin{array}{lll}
 \ud v_t-d\ud v_{xx}\leq \ud v(q-\ep- \ud v),\ \ &t\geq\tau, \ \ 0<x<h(t),\\[1mm]
 \ud v_x(t,0)=\ud v(t,h(t))=0,\ \ \ &t\geq\tau,\\[1mm]
 \ud v(\tau,x)\leq v(\tau,x),\ \ \ &0\leq x\leq h_0.
 \end{array}\right.\nonumber
\ees
On the one hand, we observe that $-\ep\leq -\frac{r u}{c+u+mv}$ for $t\geq\tau$. According to the comparison principle, we obtain
\[ \ud v(t,x)\leq v(t,x) \ \ \ {\rm for}\ t\geq\tau,\ 0\leq x\leq h(t), \]
which yields that
\[ v_x(t,h(t))\leq \ud v_x(t,h(t)). \]
On the other hand, from the proof of \cite[Lemma 3.1]{DLin}, we have
$\dd\limsup_{t\to\infty}\ud v_x(t,h(t))<0$. Hence, we have
\[\limsup_{t\to\infty}v_x(t,h(t))<0. \]
Recall that $v_x(t,h(t))=-\frac{1}{\mu}h'(t)\to 0$ as $t\to\infty$, we get a
contradiction and then, for any small $\ep\in(0,q)$,
\bess
 h_\infty\leq \frac{\pi}2\sqrt{d/(q-\ep)}.
\eess
The desired result then follows by letting $\ep\to0$.
\end{proof}

\begin{remark}
Due to $h'(t)>0$ for $t>0$, it is straightforward to show that $h_0\geq \frac{\pi}2\sqrt{d/q}$ implies $h_\infty=\infty$.
\end{remark}

The following lemma provides a sufficient condition for vanishing.
\begin{lem}\label{lemma a5.1}
Suppose that $h_0<\frac{\pi}2\sqrt{d/q}$. Then there exists $\mu_0>0$ such that $h_\infty<\infty$ provided that $\mu\leq \mu_0$.
\end{lem}
\begin{proof}
%We prove this lemma by constructing a suitable upper solution to \eqref{1.1} and then employing the comparison principle.
From \eqref{1.1}, $v$ satisfies
\bes
 \left\{\begin{array}{lll}
 v_t-dv_{xx}=v\dd\left(q-v-\frac{r u}{c+u+mv}\right),\ \ &t>0, \ \ 0<x<h(t),\\[3mm]
 v_x(t,0)=v(t,h(t))=0,\ \ \ &t\ge0,\\[1mm]
 v(0,x)=v_0(x), \ \ \ &0\leq x\leq h_0
 \end{array}\right.\label{a5.3}
 \ees
Let us construct a suitable upper solution to problem \eqref{a5.3}. Let $\delta,\alpha,M$ be positive constants to be chosen later. Define
\[ \beta(t)=h_0\left( 1+\delta-\frac{\delta}2e^{-\alpha t} \right),\ t\geq0;\ \psi(y)={\rm cos}\frac{\pi y}2,\ \ 0\leq y\leq 1, \]
\[ \bv(t,x)=Me^{-\alpha t}\psi\left(\frac{x}{\beta(t)}\right),\ \ t\geq 0,\ 0\leq x\leq \beta(t). \]
A straightforward calculation gives
\bess
 \begin{array}{lll}
 \bv_t-d\bv_{xx}-\dd\bv\left( q-v \right)\\[4mm]
 =M\dd e^{-\alpha t}\left(-\alpha\psi-\frac{x\beta'(t)}{\beta^2(t)}\psi'-\frac{d}{\beta^2(t)}\psi''-\psi\left( q-Me^{-\alpha t}\psi \right) \right)\\[4mm]
 \geq\dd Me^{-\alpha t}\psi\left(-\alpha+\frac{\pi^2}{4}\frac{d}{(1+\delta)^2h_0^2}- q+Me^{-\alpha t}\psi \right)
 \end{array}
 \eess
for all $t>0$ and $0<x<\beta(t)$. Moreover, it is easy to see that
\[\beta'(t)=\alpha h_0\frac{\delta}2e^{-\alpha t},\ \ -\mu\bv_x(t,\beta(t))=\mu Me^{-\alpha t}\frac{\pi}{2\beta(t)}. \]
Observe that $q<\frac{d\pi^2}{4h_0^2}$, we can find $\delta>0$ fulfilling
\[\frac{d\pi^2}{4(1+\delta)^2h_0^2}-q=\frac{1}2\left( \frac{d\pi^2}{4h_0^2}-q \right)>0. \]
We now choose $M$ large enough such that
\[ v_0(x)\leq M {\rm cos}\left(\frac{\pi x}{2h_0(1+\delta/2)}\right)\ \ {\rm for}\ 0\leq x\leq h_0, \]
and take
\[ \alpha=\frac{1}{2}\left( \frac{d\pi^2}{4h_0^2}-q\right),\ \ \mu_0=\frac{\delta\alpha h_0^2}{M} . \]
Then for any $0<\mu\leq\mu_0$, it holds that
\bess
 \left\{\begin{array}{lll}
 \bv_t-d\bv_{xx}\geq \bv(r-\frac{r}K\bv),\ &t>0, \ \ 0<x<\beta(t),\\[1mm]
 \bv_x(t,0)=\bv(t,\beta(t))=0,\ \ \ &t\geq0,\\[1mm]
 \beta'(t)>-\mu \bv_x(t,\beta(t)), \ \ \ &t\geq0,\\[1mm]
 \beta(0)=(1+\frac{\delta}{2})h_0>h_0.
 \end{array}\right.
 \eess
Applying the comparison principle (Lemma A.2) to get that $h(t)\leq \beta(t)$ and $v(t,x)\leq \bv(t,x)$ for $0\leq x\leq h(t)$ and $t>0$. Thereupon, we conclude that
\[ h_\infty\leq\lim_{t\to\infty}\beta(t)=h_0(1+\delta)<\infty. \]
We thus finish the proof.
\end{proof}

\begin{lem}\label{lemma a5.2}
Suppose that $h_0<\frac{\pi}2\sqrt{d/q}$.  If $ q> \frac{rM_1}{c+M_1}$ and
\[\mu\geq \mu^0:=d\max\kk\{ 1,\|v_0\|_\infty(q-\df{rM_1}{c+M_1})\rr\}\left( \frac{\pi}2\sqrt{d/q}-h_0 \right)\left(\int_0^{ h_0} v_0(x){\rm d}x\right)^{-1},\]
then $h_\infty=\infty$. Furthermore, we have that
\bess
q-\df{rM_1}{c+M_1}\leq\liminf_{t\to\infty} v(t,x)\leq\limsup_{t\to\infty} v(t,x)\leq q
\eess
uniformly for $x$ in any bounded set of $[0,\infty)$, and there exists $0<k_2\leq k_1$ such that
\bess
k_2\leq\liminf_{t\to\infty}\frac{h(t)}{t}\leq\limsup_{t\to\infty}\frac{h(t)}{t}\leq k_1,
\eess
where $k_1$ is given by Theorem \ref{t4.5}.
\end{lem}
\begin{proof}
It is easy to see that $(v,h)$ is an upper solution of
\bess
 \left\{\begin{array}{lll}
 \ud v_t-d\ud v_{xx}= \ud v\kk(q-\ud v-\df{rM_1}{c+M_1}\rr),\ \ &t>0, \ \ 0<x<\ud h(t),\\[1.5mm]
 \ud v_x(t,0)=\ud v(t,\ud h(t))=0,\ \ \ &t\ge0,\\[1mm]
 \ud h'(t)=\mu\ud v_x(t,\ud h(t)),\ \ \ &t\ge0,\\[1mm]
 \ud v(0,x)=v_0(x),\ \ \ud h(0)=h_0, \ \ \ &0\leq x\leq h_0
 \end{array}\right.
 \eess
Then we have $h(t)\geq\ud h(t)$ for $t\in[0,\infty)$ and $v(t,x)\geq\ud v(t,x)$ for $t\in[0,\infty)$ and $x\in[0,\ud h(t)]$. In view of $q- \frac{rM_1}{c+M_1}>0$, it follows from \cite[Lemma 3.7]{DLin} that $\ud h_\infty=\infty$. Hence we have $h_\yy=\yy$ by the comparison principle (Lemma A.2).

Using $\ud h_\infty=\infty$ and \cite[Theorem 4.2]{DLin}, we know that
$\lim\limits_{t\to\infty}\ud v(t,x)=q-\frac{rM_1}{c+M_1}$
uniformly for $x$ in any bounded set of $[0,\infty)$ and there exists $k_2>0$ such that
$\lim\limits_{t\to\infty}\frac{\ud h(t)}{t}=k_2$.
It follows from comparison principle (Lemma A.2) that
 \[\liminf_{t\to\infty}v(t,x)\geq\lim_{t\to\infty}\ud v(t,x)
 =q-\df{rM_1}{c+M_1} \]
uniformly for $x$ in any bounded set of $[0,\infty)$, and
\[\liminf_{t\to\infty}\frac{h(t)}{t}\geq\dd\lim_{t\to\infty}\frac{\ud h(t)}{t}=k_2.\]
Notice that $q>q-\df{rM_1}{c+M_1}$. According to \cite{DLin}, we find that $k_2\leq k_1$. This fact combined with Theorem \ref{t4.5} yields that
\[ k_2\leq\liminf_{t\to\infty}\frac{h(t)}{t}\leq\limsup_{t\to\infty}\frac{h(t)}{t}\leq k_1,\]
and
\bess
q-\df{rM_1}{c+M_1}\leq\liminf_{t\to\infty}v(t,x)\leq\limsup_{t\to\infty}v(t,x)\leq q
\eess
uniformly for $x$ in any bounded set of $[0,\infty)$.
\end{proof}

\begin{theo}\label{theorem a5.2}
Suppose that $q> \frac{rM_1}{c+M_1}$. If $h_0<\frac{\pi}2\sqrt{d/q}$, then there exists $\mu^*\geq\mu_*>0$ depending on $u_0,v_0$ and $h_0$ such that $h_\infty\leq\frac{\pi}2\sqrt{d/q}$ if $\mu\leq \mu_*$ or $\mu=\mu^*$, and $h_\infty=\infty$ if $\mu>\mu^*$
\end{theo}

\begin{proof}
(Cf. \cite[Theorem 5.2]{WZjdde17}). Writing $\Lambda:=\frac{\pi}2\sqrt{d/q}$ and $(u_\mu,v_\mu,h_\mu)$ in place of $(u,v,h)$ to clarify the dependence of solution of \eqref{1.1} on $\mu$. Define
\[ \Sigma:=\left\{ \mu>0: h_{\mu,\infty}\leq \Lambda \right\}. \]
It follows from Lemma \ref{lemma a5.1} and Theorem \ref{theorem a5.1} that $(0,\mu_0]\subset \Sigma$. According to Lemma \ref{lemma a5.2}, $\Sigma\ \cap\ [\mu^0,\infty)=\emptyset$. Hence, $\mu^*:={\rm sup}\ \Sigma\in[\mu_0,\mu^0]$. We have by this definition and Theorem \ref{theorem a5.1} that $h_{\mu,\infty}=\infty$ if $\mu>\mu^*$.

Next, we claim that $\mu^*\in\Sigma$. If not, then we would have $h_{\mu^*,\infty}=\infty$. Therefore, there is $T>0$ fulfilling $h_{\mu^*}(T)>\Lambda$. Thanks to the continuous dependence of $(u_\mu,v_\mu,h_\mu)$ on $\mu$, there exists $\varepsilon>0$ such that $h_{\mu,\infty}>\Lambda$ for $\mu\in[\mu^*-\varepsilon,\mu^*+\varepsilon]$. This entails that
\[\lim_{t\to\infty}h_\mu(t)\geq h_\mu(T)>\Lambda\ \ \ {\rm for\ all}\ \mu\in[\mu^*-\varepsilon,\mu^*+\varepsilon]. \]
Consequently, we have $[\mu^*-\varepsilon,\mu^*+\varepsilon]\cap\Sigma=\emptyset$ and then ${\rm sup}\ \Sigma\leq \mu^*-\varepsilon$. This is contrary to the definition of $\mu^*$. We have thus shown that $\mu^*\in\Sigma$.

Define
\[ \mathcal{S}=\{ \nu: \nu\geq \mu_0 \ {\rm such\ that}\ h_{\mu,\infty}\leq \Lambda\ {\rm for\ all}\ \mu\leq \nu \}, \]
where $\mu_0$ is given by Lemma \ref{lemma a5.1}. Clearly, $\mu_*={\rm sup}\ \mathcal{S}\leq\mu^*$. Similar to the above, it can be proved that $\mu_*\in \mathcal{S}$. This completes the proof.
\end{proof}

\section*{\bf Appendix}
%\subsection*{\bf A. Comparison principle}
 \def\theequation{A.\arabic{equation}}\setcounter{equation}{0}
 \newtheorem*{ALem1}{Lemma A.1}
 \newtheorem*{ALem2}{Lemma A.2}
 Let $0<T\leq \infty$ and $L_T:=(0,T)\times(0,l)$. Consider the problem
\bess
 \left\{\begin{array}{lll}
 u_t-F_1u_{xx}+F_2\gamma(u)u_x+F_3u_x=F_4\theta(u)+F_5u \ \ \ &{\rm in} \ \  L_T,\\[1mm]
 u_x(t,0)=u(t,l)=0 \ \ \ &{\rm in} \ \ [0, T),\\[1mm]
  u(0,x)=u_0(x) \ \ \ &{\rm in} \ \ [0, l],
 \end{array}\right.
 \eess
 where $\gamma,\theta\in C^1(\bar \R_+)$, $F_i(t,x)\in C(L_T)$ are bounded in $L_T$, $i=1,...,5$. Moreover, there is $\lambda>0$ such that $F_1\geq \lambda$ for $(t,x)\in L_T$.
\begin{ALem1}
Suppose that
$$\bu,\ud u\in C^{1,2}((0,T)\times[0,l])\cap C([0,T)\times[0,l]),$$
and $\bu,\ud u\geq0$ in $[0,T)\times[0,l]$. If
\bess
 \left\{\begin{array}{lll}
 \bu_t-F_1\bu_{xx}+F_2\gamma(\bu)\bu_x+F_3\bu_x +F_4\theta(\bu)+F_5\bu \\[1mm]
 \geq\ud u_t-F_1\ud u_{xx}+F_2\gamma(\ud u)\ud u_x+F_3\ud u_x+F_4\theta(\ud u)+F_5\ud u \ \ \ &{\rm in} \ \  L_T,\\[1mm]
 \bu_x(t,0)\leq \ud u_x(t,0),\ \ \bu(t,l)\geq \ud u(t,l) \ \ \ &{\rm in} \ \ [0,T),\\[1mm]
  \bu(0,x)\geq \ud u(0,x) \ \ \ &{\rm in} \ \ [0, l],
 \end{array}\right.
 \eess
then $\bu\geq\ud u$ in $[0,T)\times[0,l]$.
\end{ALem1}
\begin{proof}
Let $w=\bu-\ud u$, then we have
\bess
 \left\{\begin{array}{lll}
 w_t-F_1w_{xx}+F_2\gamma(\ud u)w_x+F_2\bu_x[\gamma(\bu)-\gamma(\ud u)]\\[1mm]
 \qquad +F_3w_x+F_4[\theta(\bu)-\theta(\ud u)]+F_5w \geq 0 \ \ \ &{\rm in} \ \ L_T,\\[1mm]
 w_x(t,0)\leq 0,\ \ w(t,l)\geq 0 \ \ \ &{\rm in} \ \ [0,T),\\[1mm]
  w(0,x)\geq 0 \ \ \ &{\rm in} \ \ [0, l].
 \end{array}\right.
 \eess
Writing
\bess
\gamma(\bu)-\gamma(\ud u)=(\bu-\ud u)\int_0^1\gamma'(\ud u+s(\bu-\ud u)){\rm d}s,\\
\theta(\bu)-\theta(\ud u)=(\bu-\ud u)\int_0^1\theta'(\ud u+s(\bu-\ud u)){\rm d}s.
\eess
Denote $a_2(t,x)=F_2(t,x)\gamma(\ud u)+F_3(t,x)$ and
\begin{align*}
 a_3(t,x)=&F_5(t,x)+F_2(t,x)\bu_x\int_0^1\gamma'(\ud u+s(\bu-\ud u)){\rm d}s+F_4(t,x)\int_0^1\theta'(\ud u+s(\bu-\ud u)){\rm d}s.
\end{align*}
Then $a_3$ is bounded in $[\ep,T-\ep]\times[0,l]$ for each fixed $0<\ep\ll1$, and $w$ satisfies
\bess
 \left\{\begin{array}{lll}
w_t-F_1(t,x)w_{xx}+a_2(t,x)w_x+a_3(t,x)w \geq 0 \ \ \ &{\rm in} \ \ L_T,\\[1mm]
 w_x(t,0)\leq 0,\ \ w(t,l)\geq 0 \ \ \ &{\rm in} \ \ [0,T),\\[1mm]
  w(0,x)\geq 0 \ \ \ &{\rm in} \ \ [0, l].
 \end{array}\right.
 \eess
The standard maximum principle yields $w\ge 0$, and so $\bu\geq\ud u$.
\end{proof}

\begin{ALem2}
Suppose that $T>0$, $\bh\in C^1([0,T])$ and $\bh>0$ in $[0,T]$, $\bv \in C(\bar D_T)\cap C^{1,2}(D_T)$ with $D_T:=\{ 0<t\leq T,0<x<\bh(t) \}$, $\psi$ is $C^1$ and satisfies $\psi(0)=0$, $(\bv,\bh)$ satisfies
 \bess
 \left\{\begin{array}{lll}
 \bv_t-\bv_{xx}\geq \psi(\bv),\ \ &0<t\leq T, \ \ 0<x<\bh(t),\\[1mm]
 \bv_x(t,0)= 0,\ \ \bv(t,\bh(t))=0,\ \ \ &0\leq t\leq T,\\[1mm]
 \bh'(t)\geq-\mu \bv_x(t,\bh(t)), \ &0\leq t\leq T.
 \end{array}\right.
 \eess
If $\bh(0)\geq h_0$, $\bv(0,x)\geq 0$ in $[0,\bh(0)]$, and $\bv(0,x)\geq v_0(x)$ in $[0,h_0]$. Then $h(t)\leq \bh(t)$ in $(0,T]$, and $v(t,x)\leq \bv(t,x)$ for $t\in[0,T]$ and $x\in[0,h(t)]$.
\end{ALem2}

The proof of Lemma A.2 is identical to that of \cite[Lemma 3.5]{DLin}.

\end{document}